\documentclass{amsart}

\usepackage{amssymb}
\usepackage[stretch=10,shrink=10]{microtype}
\usepackage[showframe=false, margin = 1.5 in]{geometry}
\usepackage{mathtools}
\usepackage{tikz}
\usepackage{subfigure}
\usepackage{comment}
\usepackage{mathrsfs}

\tikzset{%
  raindrop/.pic={
    code={\tikzset{scale=1/10}
 \clip [preaction={top color=blue!50!cyan!50, bottom color=blue!50!cyan}]
 (0,0)  .. controls ++(0,-1) and ++(0,1) .. (1,-2)
 arc (360:180:1)
 .. controls ++(0,1) and ++(0,-1) .. (0,0) -- cycle;
 \foreach \j in {1,3,...,20}
 \shade [top color=blue!50!cyan!50, shift=(270:0), xscale=1-\j/40,yscale=1-\j/80, white, opacity=1/15]
 [rotate=-\j] (0,0)  .. controls ++(0,-1) and ++(0,1) .. (1,-2)
 arc (360:180:1)
 .. controls ++(0,1) and ++(0,-1) .. (0,0) -- cycle;
  }}}

\newcommand{\no}{\nonumber}
\newcommand{\bi}{\begin{itemize}}
\newcommand{\ei}{\end{itemize}}

\newcommand{\E}{\mathbf{E}}
\renewcommand{\P}{\mathbf{P}}

\newcommand{\ind}[1]{\mathbf{1}{\{ #1 \}}}

\renewcommand{\le}{\leq}
\renewcommand{\ge}{\geq}
\makeatletter
\newcommand*\proc{{\mathpalette\bigcdot@{.7}}}
\newcommand*\bigcdot@[2]{\mathbin{\vcenter{\hbox{\scalebox{#2}{$\m@th#1\bullet$}}}}}
\makeatother

\DeclareMathOperator{\Geo}{Geo}

\newtheorem{thm}{Theorem}
\newtheorem{lemma}[thm]{Lemma}
\newtheorem{prop}[thm]{Proposition}

\theoremstyle{remark}

\theoremstyle{definition}

\renewcommand{\vec}{\mathbf }

\newcommand{\commentout}[1]{}

\author[Cristali]{Irina Cristali}
\address{Department of Mathematics, Duke University}
\email{irina.cristali@duke.edu}

\author[Ranjan]{Vinit Ranjan}
\email{vinit.ranjan@duke.edu}

\author[Steinberg]{Jake Steinberg}
\email{jake.steinberg@duke.edu}

\author[Beckman]{Erin Beckman}
\email{ebeckman@math.duke.edu}

\author[Durrett]{Rick Durrett}
\email{rtd@math.duke.edu}

\author[Junge]{Matthew Junge}
\email{jungem@math.duke.edu}
 
\author[Nolen]{James Nolen}
\email{nolen@math.duke.edu}

\begin{document}

\title{Block size in Geometric($p$)-biased permutations}

\maketitle

\begin{abstract}
Fix a probability distribution $\mathbf p = (p_1, p_2, \hdots)$ on the positive integers. The first block in a $\mathbf p$-biased permutation can be visualized in terms of raindrops that land at each positive integer $j$ with probability $p_j$. It is the first point $K$ so that all sites in $[1,K]$ are wet and all sites in $(K,\infty)$ are dry. 
For the geometric distribution $p_j= p(1-p)^{j-1}$ we show that $p \log K$ converges in probability to an explicit constant  as $p$ tends to 0. Additionally, we prove that if $\mathbf p$ has a stretch exponential distribution, then $K$ is infinite with positive probability.  

\end{abstract}

\section{Introduction}

Let $\vec p = (p_1, p_2,\hdots)$ be a discrete probability distribution  on the positive integers with each $p_j > 0$, and $\mathbf X = (X_1,X_2,\hdots)$ be a sequence of independent $\vec p$-distributed random variables. 
%
 A \emph{$\vec p$-biased permutation} is a map $i \to \Pi_i$ where  $(\Pi_1 , \Pi_2 , \hdots)$ is the sequence of distinct values in order of appearance from $\mathbf X.$ For example, if $\mathbf X = (3,1,4,1,3,1,2, \hdots)$, the beginning of $\Pi$ is 
\begin{center}
\begin{tabular}{ccccc}
$i$ & 1 & 2 &  3 & 4\ $ \cdots$ \\
$\Pi_i$ & 3 & 1 & 4\ & 2 $\cdots$
\end{tabular}.
\end{center}
The \emph{first block} in $\Pi$ is the smallest interval $\{1,\hdots, K\} := [K]$ such that $\Pi([K]) = [K]$. We call $K$ the \emph{block size}, because it is the first point at which the restriction $\Pi|_{[k]}$ becomes an element of the symmetric group $S_k$, i.e.\
$$K = \inf \{ k \colon \Pi |_{[k]} \in S_k\}.$$
Also, define $N$ to be the minimal number of samples so that $\{X_1,\hdots, X_N\} = [K]$. In the example above we have $K=4$ and $N=7$.  

One way to visualize the formation of blocks is by imagining rain falling on a partitioned stick. The \emph{$\mathbf p$-rainstick process} starts with a partition of the unit interval where the $j$th segment ordering from left to right has size $p_j>0$. 
Raindrops fall on points chosen uniformly at random from the unit interval.  A segment starts out dry and becomes wet after a raindrop falls on it. We can think of $X_n$ as the index of the segment hit by the $n$th raindrop. Thus, the $\mathbf p$-rainstick process can be thought of as occurring on $[0,1]$, or on the positive integers. In this setting, $N$ is the first time that, for some $K$,  $[1,K]$ is wet and $(K,\infty)$ is dry. In this paper, we focus mostly on the case that $\mathbf p$ is the geometric distribution with parameter $p$, i.e.\ $p_j = p(1-p)^{j-1}$. We will denote this distribution by $\Geo(p)$. 
Our main interest is the asymptotic behavior of $K$ and $N$ as $p \to 0$.

Rare events greatly influence the behavior of $K$ for the $\Geo(p)$-rainstick process. If a segment far from 0 becomes wet, then, in order to form a block, rain must fall on all of the dry segments to its left.  When $p$ is small, this may take a long time. Meanwhile another more distant segment may be hit, and so on.  It turns out that dry sites persist behind the maximal wet site for a long time, and this persistence causes $K$ to grow like $e^{b/p}$ as $p \to 0$. Remarkably, we are able to calculate the constant $b$ exactly. 

\commentout{
It was initially unclear to us how $K$ should diverge to infinity as $p\to0$. 
A lower bounding process described in Lemma \ref{lem:lower1} halts once the point immediately to the left of the current maximum becomes wet. This process terminates with the rightmost wet point at distance $\approx p^{-2}$. We discovered this early on in the project, and believed for quite some time that the number of dry sites behind the front was $O(1)$. So, we thought $K$ grew like a polynomial in $p^{-1}$. 
}

The $\Geo(p)$-rainstick process is a particular case of a random allocation model known as the \emph{Bernoulli sieve}.  In that model, a ball is allocated to the $j^{th}$ bin with a probability $p_j$ given by random stick breaking,
\begin{align} 
p_j = (1- W_1)\cdots (1- W_{j-1}) W_i, \label{eq:Wi}
\end{align}
with $0<W_j <1$ and the $W_j$'s being i.i.d. See \cite{BS,BS2} for an overview. The non-random case $W_j \equiv p$ corresponds to the $\Geo(p)$-rainstick process. The literature focuses on different statistics for the counts of balls in various bins \cite{bs3,bs4}, such as the number of bins with exactly $r$ balls, or the location of the rightmost filled bin. To our knowledge, the relationship between the Bernoulli sieve and block sizes of random permutations has not been studied. 

The family of $\Geo(p)$-biased permutations is a standard example of \emph{regenerative permutations} which were recently introduced by Pitman and Tang in \cite{pitman}. Roughly speaking, regenerative permutations are infinite permutations $\Pi$ whose blocks $\Pi( [K_n] ) = [K_n]$ have i.i.d.~increments, $K_n-K_{n-1}$. 
%
\cite[Proposition 1.7]{pitman} establishes that a  $\Geo(p)$-biased permutation has $\E K < \infty$ for all $p$.  
A  result worth mentioning from Duchamps et.\ al.\ pertains to $W_i=\text{Uniform}(0,1)$ with $W_i$ from \eqref{eq:Wi}. \cite[Proposition 1.1]{pitman2} gives the surprisingly simple formulas $\E K = 3$ and $\text{var}(K) = 11.$
The proof relates renewal probabilities to linear combinations of the Riemann-zeta function. Pitman suggested that it would be interesting to extend this formula to $W_i = \text{Beta}(\theta)$, but so far the uniform case ($\theta =1$) is the only one that has been solved exactly [private correspondence]. Finding exact values of $\E K$ in our case $W_i \equiv p$ would also be interesting. We put some effort toward this, but had no success.

Blocks in infinite permutations are useful for studying the asymptotic behavior of finite  permutations. In \cite{naya}, Basu and Bhatnagar prove a central limit theorem for the longest increasing subsequences in Mallows($1-p$) distributed permutations for small $p$. See \cite{gnedin} for a definition and account of the limiting behavior of these permutations.
 To do this they decompose a related regenerative permutation, called a $\Geo(p)$-\emph{shifted permutation}, into blocks, and then concatenate the longest increasing sequences from each block.

Block formation in $\Geo(p)$-shifted permutations can be described via what we will call the \emph{paintstick process}. Imagine that on the $i$th step, a paintball falls on $X_i \sim \Geo(p)$. All of the integers left of $X_i$ are painted red to indicate that they must all be used before the block is formed. The integer at $X_i$ is removed, and the values to its right shift to the left by 1.  The moment $K'$ when there are no red sites is when the first block has been created. Unlike the rainstick process, there is no redundancy. Hence the time to form a block and the block size are both equal to $K'$. 

A key estimate in \cite{naya} is that $p\log \E K'\to \pi^2/6$ as $p \to 0$. So the expected block size grows exponentially as $p \to 0$. The proof is mainly concerned with the running maximum of geometric samples. Leader-election procedures are another setting in which this quantity is important \cite{geo1,geo2}. In the paintstick process, the painted sites always form an interval -- this makes the analysis of the paintstick much simpler than that of the rainstick process, in which one has to understand the structure of the dry sites.

\subsection*{Overview of results}
Our first three results  describe the size of the first block in $\Geo(p)$-biased permutations.  Let $b = \log 2 - \int_0^\infty  \log (1 - 2^{-e^{y}}) \, dy \approx 1.1524.$
%

\begin{thm}\label{thm:KU}
The block size $K$ is stochastically dominated by a $\Geo( p e^{- b/p})$ random variable. In particular, $\E  K \leq (1/p) e^{b/p}$ and for all $\epsilon > 0$
$$
\P[ K > e^{(b + \epsilon)/p}]\to 0, \quad \text{as} \;\; p \to 0.
$$ 
\end{thm}

The proof of this bound will be described in terms of the $\Geo(p)$-rainstick process, and the idea of the proof may be summarized as follows: We first extend the rainstick process to all of the integers, with rain falling on $j \in \mathbb Z$ in continuous time at rate $(1-p)^{j-1}p$. 
Let $M_t$ be rightmost wet point at time $t$. To obtain our upper bound we make the drastic modification that, after the maximum $M_t$ increases, all of the sites in $(-\infty,M_t)$ are reset to being dry. Let $G$ be the event that all sites $j < M_t$ become wet before the boundary moves again. We then show that $\P[G] \geq e^{-b/p}$. Thus, the number of tries we need to have the event $G$ occur is stochastically dominated by a $\Geo(e^{-b/p})$ random variable.  The factor of $p$ in the stochastic bound is because the boundary increases by a $\Geo(p)$-distributed amount at each increase. 

Given the drastic simplification in the last paragraph, it is surprising that the constant $b$ is sharp, as shown by this lower bound:

\begin{thm}\label{thm:KL}
For all $\epsilon > 0$, $\P[ K >  e^{(b - \epsilon)/p}] \to 1$ and hence $p \log K \to b$ in probability, as $p \to 0$.
\end{thm}

\noindent This lower bound is more difficult to prove than the upper bound of Theorem \ref{thm:KU}. We do so by getting an upper bound on $\P[K=k]$ by estimating the probability of the event 
$$
G_{k,t} = \{\text{all sites $\le k$ are wet and all sites $>k$ are dry at time $t$}\}.
$$ 
One can write an explicit formula for $\P[G_{k,t}]$, see \eqref{pdoubleprod}. Bounding $\int_0^\infty \P[G_{k,t}] \, dt$ then
leads to the desired result.

Since the maximum of $n$ independent $\Geo(p)$ random variables grows logarithmically in $n$, it follows easily from these two theorems that $N$ grows as a double exponential:

\begin{thm} \label{cor:N}
$p \log \log N \to b$ in probability as $p \to 0$.
%
\end{thm}

\noindent   
When $p=0.1$, $\exp(e^{b/p}) \approx 10^{42,000}$, so it seems unlikely that one could predict the limiting behavior of $K$ by simulation.

Since the time grows doubly exponentially fast for the geometric as $p\to 0$, it should not be surprising that when the tail of the distribution $\vec p$
is stretched exponential there is positive probability that the process never terminates.  

\begin{thm} \label{prop:stretch}
Fix $\alpha \in (0,1)$ and let $K_\alpha$ be the size of the first block in a $\vec p$-biased permutation where $\P[X_1 \geq k] =  C_\alpha e^{-k^\alpha}$ with $C_\alpha$ a normalizing constant. It holds that
$$
\P[K_\alpha = \infty ] >0.
$$ 	
\end{thm}


\section{Proofs}

We adopt a continuous time perspective for the rate at which raindrops fall. The advantage is that, by Poisson thinning, raindrops arrive at each integer as independent Poisson processes. Let $M_t$ denote the maximum point in the process at time $t$ with $M_0 = X_1$. We will often rescale time so raindrops fall at unit intensity at $M_t$. 
We will refer to integers left of $M_t$ yet to be covered as \emph{dry sites}, and covered integers as \emph{wet sites}. Accordingly, let $H_t$ be the number of dry sites in $[1,M_t).$ Set $\eta = \inf\{ t \geq 0 \colon H_t = 0\}$ so that $K = M_\eta.$

\subsection{Proof of Theorem  \ref{thm:KU}} \label{sec:KU}

Given $M_t = m$, we may rescale time so that raindrops arrive at $m$ at rate $1$, and raindrops arrive at site $m + \ell$ at rate
\begin{align}(1 - p)^\ell,\quad  \text{ for } \ell \in [1,\infty).\label{eq:l}
\end{align}
After this rescaling, a new maximal wet site arrives at rate
\begin{align}\textstyle{\sum_{\ell \geq 1} (1 - p)^{\ell} = p^{-1}(1-p)}. \label{eq:max}	
\end{align}

If we allow for dry sites all the way to $-\infty$, then we can extend the rate in \eqref{eq:l} to all $\ell \in \mathbb Z$. Observe that by independence we have $M_t$ is unchanged when we allow for rain to arrive at the non-positive integers as well. To bound $H_t$ from above, we consider a ``forgetful" version of this process: whenever $M_t$ increases we reset all of the sites in $(-\infty, M_t)$ to be dry. This modification has no effect on $M_t$. Let $\hat H_t$ be the number of dry sites in $(-\infty, M_t)$ in this forgetful process.  The natural coupling ensures that $H_t \leq \hat H_t$, so it takes longer for $\hat H_t$ to reach 0 than its counterpart $H_t$.  Thus, $\hat \eta := \inf\{ t \colon \hat H_t = 0\} \succeq \eta.$ It follows that 
\begin{align}M_{\hat \eta} \succeq M_\eta = K. \label{eq:hatM}
\end{align}
We now show that the probability that all of the sites behind $M_t$ become wet before $M_t$ increases is at least $e^{- b /p}$.

\begin{lemma}\label{lem:prob2}
Let $t_1$ to be the first time that $M_0$ is exceeded. For all $p \in (0,1)$ we have
\begin{align}
\P[ \hat \eta <t_1 ] \ge  e^{-b/p}.
\label{plb}
\end{align}
\end{lemma}

\begin{proof} 
Because of the time-scaling described at \eqref{eq:l} and \eqref{eq:max}, the arrival time $t_1$ is exponential with rate $p^{-1}(1 - p)$. Conditioning on $t_1$ gives
$$
\P[ \hat \eta <t_1 ] = \frac{1-p}{p} \int_0^\infty e^{-s(1-p)/p} \prod_{\ell=1}^\infty \left( 1 - \exp(-s(1-p)^{-\ell}) \right) \, ds  .
$$
Let $\alpha = \log 2$.  The right side may be bounded below by restricting the domain of integration to $[\alpha, \infty)$. Over this domain, $1 - \exp(-s(1-p)^{-\ell}) \geq 1 - \exp(-\alpha(1-p)^{-\ell})$, so
$$
\P[ \hat \eta <t_1 ]  \geq e^{-\alpha/p} \prod_{\ell=1}^\infty ( 1 - \exp(-\alpha(1-p)^{-\ell}).
$$ 
Taking the log of the infinite product and changing variables $l=x/p$ we have
$$
\log \P[ \hat \eta <t_1 ] \geq - \frac{\alpha}{p}+  \frac{1}{p} \cdot p \sum_{x \in p\mathbb N} \log (1- \exp(-\alpha (1-p)^{-x/p}) ).
$$
The fact that $1-p \le e^{-p}$ implies $(1-p)^{-x/p} \ge e^{x}$. This gives $\exp(-\alpha (1-p)^{-x/p}) \le \exp(-\alpha e^{x})$. Hence,
$$
\log \P[ \hat \eta <t_1 ] \geq  - \frac{\alpha}{p} + \frac{1}{p}  \sum_{x \in p\mathbb N} p \log (1- \exp(-\alpha e^{-x}) ).
$$
The last sum is a Riemann sum that evaluates the function at the right endpoint of each interval. Since the integrand is increasing, the sum is larger than
$$
\int_0^\infty  \log (1 - \exp(-\alpha e^{y})) \, dy.
$$
Therefore,
$$
\log \P[ \hat \eta < t_1 ] \ge  -\frac{\log 2}{p} + \frac{1}{p} \int_0^\infty  \log (1 - \exp(-(\log 2) e^{y})) \, dy,
$$
which establishes the lemma.
\end{proof}

We can now quickly deduce Theorem \ref{thm:KU}. 

\begin{proof}[Proof of Theorem \ref{thm:KU}]
Recall that \eqref{eq:hatM} gives $K \preceq  M_{\hat \eta}.$ By Lemma \ref{lem:prob2} we have all of the dry sites become wet with probability at least $e^{-b/p}$. However, if the maximum increases, it does so by a $\Geo(p)$ distributed amount. When this occurs we reset all of the sites behind the new maximum to be dry, thus restarting the dynamics. It follows that
$$M_{\hat \eta} \preceq \sum_{i=1}^{\Geo(e^{- b/p}) } \Geo(p) \overset{d} = \Geo( p e^{- b /p} ).$$
The terms are independent because the amount the maximum increases by $\Geo(p)$ does not depend on how long it takes to cover all of the dry sites behind it. This is because we reset all sites $(-\infty, M_n)$ to be dry each time the maximum increases. 
 In light of \eqref{eq:hatM} we then have the claimed dominance
$K  \preceq  M_{\hat \eta} \preceq \Geo( p e^{- b/p}).$
\end{proof}

\subsection{Proof of Theorem \ref{thm:KL}} \label{sec:LU}

We need to show that for any $\epsilon > 0$, $\P[K \leq e^{(b - \epsilon)/p}] \to 0$ as $p \to 0$.  We consider separately the possibility of small and large realizations of $K$. First we show that small values of $K$ are unlikely.   

\begin{lemma} \label{lem:lower1}
$\P[K > 2/p^{3/2}] \to 1$ as $p \to 0$.
\end{lemma}

\begin{proof}
 We obtain a lower bound on $K$ if we halt the process the first time the point immediately left of $M_t$ is filled. To get started the first raindrop must land beyond 1. This happens with probability $1-p$. After this, the maximum will jump by a $1+\Geo(p)$ distributed amount a $\Geo(q)$-distributed number of times. Here $q$ is the probability that a rate-$(1-p)^{-1}$ exponential random variables is smaller than a rate-$p^{-1}(1-p)^2$ exponential random variable.   These rates come from \eqref{eq:l}. It follows that
$$\ind{X_1 >1} \sum_{k=1}^{\Geo(q)} (1+ \Geo(p)) \preceq K.$$
It is straightforward to compute that $q \sim p$ and $\P[X_1 >1] = 1-p$. So, the left side converges in distribution as $p \to 0$ to a random variable that is stochastically larger than a $\Geo(p^2)$ random variable.  The result follows from the observation that $$\P[\Geo(p^2) > 2p^{-3/2}]= (1-p^2)^{2p^{-3/2}} \approx e^{-2p^{1/2}} \to 1$$
as $p \to 0$.  
\end{proof}

Our main estimate is an exponential bound for larger values of $K$.

\begin{prop} \label{prop:lower2}
For any $\epsilon > 0$, we have
\[
 \max_{k \geq 2/p^{3/2}} \P[K=k] \le e^{- (b - \epsilon)/p}
\]
if $p$ is sufficiently small.
\end{prop}

With these two estimates it is elementary to prove Theorem \ref{thm:KL}.

\begin{proof}[Proof of Theorem \ref{thm:KL}]
Start by decomposing $\P[K \leq e^{(b - \epsilon)/p}]$ as   
\begin{align}
 \P[K \leq e^{(b - \epsilon)/p}] & =  \P[ K < 2/p^{3/2} ]+ \sum_{k = 2/p^{3/2}  }^{ e^{(b - \epsilon)/p}} \P[K = k]. \no
\end{align}
The first term vanishes as $p \to 0$, by Lemma \ref{lem:lower1}. The second term is bounded by 
\[
e^{(b - \epsilon)/p} \max_{k \geq 2/p^{3/2}} \P[K = k],
\]
which also vanishes, by Proposition \ref{prop:lower2} with $\epsilon' = \epsilon/2$.
\end{proof}

It remains to prove Proposition \ref{prop:lower2}. During the proof we will state two lemmas, which we establish immediately after.

\begin{proof}[Proof of Proposition \ref{prop:lower2}.]

In what follows we write $u_p$ as shorthand for $\exp(-b/p + o(1/p))$. The constants implicit in $o(1/p)$ may change from line to line, but in every instance this quantity is independent of $k \geq 2/p^{3/2}$. We will run the process at rate $1/[p(1-p)^{(k-1)}]$. 

Let $G_{j,t}$ be the event that at time $t$ (in this new time scale) all sites less than or equal to $j$ are wet and all larger than $j$ are dry. Call such a formation a \emph{$j$-block}.  We claim that
\begin{equation}
\P[K=k] \leq \frac{1-p}{p} \int_0^\infty \P[ G_{k,t} ] \, dt. \label{Kkintegral}
\end{equation}
With this new time scaling, drops fall on site $j$ at rate $(1 - p)^{j-k}$, and drops fall in the region $[k+1,\infty)$ at rate 
\[
\frac{1}{p(1 - p)^{k-1}} \sum_{j = k+1}^\infty p(1 - p)^{j-1} = \frac{1-p}{p}.
\]
Therefore, given that a block forms at $k$, the block will have an expected lifetime of $\frac{p}{1 - p}$ before a drop falls in $[k+1,\infty)$. The bound \eqref{Kkintegral} follows from this: 
\begin{align}
\int_0^\infty \P[ G_{k,t} ] \, dt & = \E \left[ \int_0^\infty \ind{ G_{k,t}} \,dt \right] \no \\
& = \E[ \text{lifetime of a $k$-block}\;|\; \text{a $k$-block forms}] \P[\text{a $k$-block forms}] \no \\
& = \frac{p}{1-p} \P[\text{a $k$-block forms}] \geq \frac{p}{1-p} \P[ K = k]. \no 
\end{align}

To bound the right-hand side of \eqref{Kkintegral}, we note that by Poisson thinning
\begin{align}
\P[G_{j,t} ] & = \prod_{m=1}^j \left[ 1- \exp(- t(1-p)^{(m-k)}) \right]
\cdot \prod_{m=j+1}^\infty  \exp(- t(1-p)^{(m-k)}). \label{pdoubleprod}
\end{align}

At any time $t$ there is an integer $j(t)$ that maximizes $\P[G_{j,t}]$ over all $j \geq 0$.  Note that $j(t)$ implicitly depends on $k$. We describe how $\P[G_{j(t),t}]$ behaves in the following two lemmas.

\begin{lemma} \label{lem:lower3}
If $t_1 = \exp(-1/ (2\sqrt{p})) \log 2$, then
\[
\max_{t \geq t_1} \P[ G_{j(t),t}] \leq u_p.
\] 
If $t = \log 2$, then $j(t) = k$. 
\end{lemma}

\begin{lemma} \label{lem:lower4}
Let $t_0= 3 \log 2$. For any $n > 1$,
$$
\frac{\P[ G_{k,t}]}{\P[ G_{j(t), t}]} \le t^{-n} \quad\hbox{for all $t \ge t_0$.}
$$ 
holds if $p$ is sufficiently small (depending on $n$, but not on $k$).
\end{lemma}

The combination of Lemma \ref{lem:lower3} and Lemma \ref{lem:lower4} allows us to control the integral in \eqref{Kkintegral} from $t_1$ to $\infty$:
\[
\int_{t_1}^\infty \P[ G_{k,t}] \,dt \leq u_p (t_0  - t_1) +  u_p \int_{t_0}^\infty t^{-n} \,dt.
\]
To control the integral over small times $t < t_1 < \epsilon$, we use the bound
$$
\P[G_{k,t}] \le  \prod_{m=k - 1/p}^k \left[ 1- \exp(- t(1-p)^{(m-k)}) \right] \le [ 1- \exp(-\epsilon e )]^{1/p} \le e^{-2/p}
$$
for $\epsilon$ is small enough. 
Since $b < 2$ and since $t_1 < \epsilon$ for $p$ small, it follows that $$\int_0^{t_1} \P[ G_{k,t} ] \, dt \le \epsilon u_p$$
and the proof of Proposition \ref{prop:lower2} is complete.
\end{proof}

Now we need to establish the two lemmas used in the proof of Proposition \ref{prop:lower2}.

\begin{proof}[Proof of Lemma \ref{lem:lower3}.]
For fixed $t$, define the real number $j^* = j^*(t)$ by $t(1 - p)^{j^* - k} = \log 2$. From \eqref{pdoubleprod}, we see that the ratio
\[
\frac{\P[G_{j,t}]}{\P[G_{j-1,t}]} = \frac{1 - \exp( - t(1 - p)^{j - k})}{\exp( - t(1 - p)^{j - k})}
\]
is greater than $1$ if $j < j^*$ and it is less than $1$ if $j > j^*$. If $j^* \geq 1$, this shows that $\P[G_{j,t}]$ is maximized at 
	$$j(t) = \max \{ j \in \mathbb Z \colon j \leq j^*\}.$$ 
 Otherwise, $\P[G_{j,t}]$ is maximized at $j = 0$.  Observe that $j^*(t) = j(t) = k$ when $t=\log 2$. The fact that $t(1 - p)^{j^* - k} = \log 2$ implies  $p(j(t) - k) = \log(t/\log 2) + O(p)$ as $p \to 0$.

Now we estimate $\P[G_{j(t),t}]$, assuming $t \geq t_1$. Recall that $k \geq 2p^{-3/2}$ is assumed in Proposition \ref{prop:lower2}.  This and the assumption that $t \geq t_1$ guarantees that $j(t) \geq p^{-3/2}$.  By \eqref{pdoubleprod} we have
\begin{align}
p \log \P[G_{j(t),t}] & = p\sum_{1 \leq m \leq j(t)} \log \left(1 - \exp( - t(1 - p)^{m - k}) \right) - p \sum_{m > j(t)} t ( 1 - p)^{m-k}  \no \\
& =  p\sum_{1 \leq m \leq j(t)} \log \left(1 - \exp( - t(1 - p)^{m - k}) \right) - t (1 - p)^{j(t)+1 - k}.
\end{align}
The last term is equal to $- (1 - p)^{1 + j(t) - j^*(t)} \log 2$ which converges to $-\log 2$ as $p \to 0$, uniformly in $k$, since $- 1 \leq j(t) - j^*(t) \leq 0$.  The first sum is a Riemann sum approximating an integral.  Let us write the sum as
\begin{align}
 p\sum_{ 1 \leq m \leq j(t)} \log \left(1 - \exp( - (1 - p)^{m - k + \frac{\log(t)}{\log(1 - p)}}) \right) & = p\sum_{ \ell \in I_t} \log \left(1 - \exp( - (1 - p)^{\ell/p}) \right)
\end{align}
where the index set is 
\[
I_t = \left\{ \ell +  p \frac{\log t}{\log(1 - p)} \colon  \ell \in p \mathbb{Z} , \quad p(1 - k) \leq \ell \leq p(j(t) - k)  \right\}.
\]
Since $p(j(t) - k) = \log(t/\log 2) + O(p)$, the upper index in the sum is
\begin{align*}
 p(j(t) - k) +  p \frac{\log t}{\log(1 - p)} &= p(j^*(t) - k) +  p \frac{\log t}{\log(1 - p)} + p(j(t) - j^*(t)) \\ 
 	&= - \log(\log 2) + O(p).
\end{align*}
This converges to $- \log\log2$ as $p \to 0$, uniformly over $t \geq t_1$ and $k \geq 2 p^{-3/2}$.  The lower index in $I_t$ is
\[
p(1 - k) +  p \frac{\log t}{\log(1 - p)} \leq p -  2p^{-1/2}  +  p \frac{\log t_1}{\log(1 - p)},
\]
which converges to $-\infty$ as $p \to 0$, also uniformly over $t \geq t_1$ and $k$.  Consequently,
\begin{align}
\lim_{p \to 0} p \log \P[G_{j(t),t}] & = \int_{-\infty}^{- \log(\log 2)} \log(1 - \exp( - e^{-r})) \,dr - \log 2 \no \\
& = \int_{0}^{\infty} \log \left( 1 - \exp( -  e^{y} \log 2) \right) \,dy - \log 2 \no \\
 &= b, \no
\end{align}
and the convergence is uniform over $t \geq t_1$ and $k \geq 2p^{-3/2}$. The final equality 
gives the claimed bound on $\P[G_{j(t),t}]$.
\end{proof}

\begin{proof} [Proof of Lemma \ref{lem:lower4}.]

By the definition of $j^*(t)$ and $j(t)$ in Lemma \ref{lem:lower3}, $j(t) > k$ whenever $t > (1 - p)^{-1} \log(2)$. So, $t \geq t_0$ guarantees that $j(t) > k$, assuming $p < 1/2$. In fact, $(j(t) - k) = p^{-1} \log(t/\log 2) + O(1)$ as $p \to 0$.  
Since $j(t) > k$, from \eqref{pdoubleprod} we see that
\begin{align}
\frac{\P( G_{k,t} )}{ \P( G_{j(t),t} )}  =  \prod_{m=k+1}^{j(t)} 
\frac{  \exp(- t(1-p)^{(m-k)}) } { 1- \exp(- t(1-p)^{(m-k)})  }  =  \prod_{\ell=1}^{j(t)-k} \frac{  \exp(- t(1-p)^{\ell}) }{   1- \exp(- t(1-p)^{\ell})   }. \label{Pratio2}
\end{align}
 The ratios in the last product are increasing with respect to $\ell$, and are all bounded by $1$ (since $\ell \leq j(t) - k$, and by definition of $j(t)$).   Define $\ell^*(t) \in [1,j(t) - k)$ by $t(1 - p)^{l^*} = 1$. If $p$ is small enough, such $\ell^*$ exists for all $t \geq t_0 > 1$. Then for all $\ell \leq \ell^*$ we have
\[
\frac{  \exp(- t(1-p)^{\ell}) }{  1- \exp(- t(1-p)^{\ell})  } \leq \frac{  \exp(- t(1-p)^{\ell^*}) }{  1- \exp(- t(1-p)^{\ell^*})  } = \frac{e^{-1}}{1 - e^{-1} } < 1.
\]
Hence, the product in \eqref{Pratio2} is bounded by
\begin{align}
\frac{\P[ G_{k,t} ]}{ \P[ G_{j(t),t} ]}  \leq \left( \frac{e^{-1}}{1 - e^{-1} } \right)^{\ell^*}.
\end{align}
By definition, $\ell^* = -\log(t)/\log(1 - p)$. So, with $\gamma = -\log \left(\frac{e^{-1}}{1 - e^{-1} }\right) > 0$, this is
\[
\frac{\P[ G_{k,t} ]}{ \P[ G_{j(t),t} ]} \leq e^{-\gamma \ell^*(t)} = t^{-\gamma/|\log (1 - p)|}.
\]
For any $n > 1$, $\gamma/|\log (1 - p)| > n$ if $p$ is small enough, depending on $n$ but not on $k$.

\end{proof}

\subsection{Proof of Theorem \ref{cor:N}} \label{sec:N}

\begin{proof} 
We begin by noting that, by having raindrops fall according to a Poisson point process, we have \emph{Poissonized} the rainstick process. This makes $N$ into a real-valued, rather than integer-valued, random variable. \emph{Depoissonizing} back to the integer case is standard (see \cite{poisson}), and we omit those details.

Let $\beta<\alpha$. If $n = \exp(e^{\beta/p})$ the probability that a site beyond $e^{\alpha/p}$ becomes wet before time $n$ is
\begin{align}
1 - \exp(-n(1-p)^{e^{\alpha/p}}) = 1 - \exp\left(-\exp(e^{\beta/p}) \cdot \exp(\log(1-p) e^{\alpha/p}) \right) \to 0 \label{alphabetalim}
\end{align}
as $p \to 0$. Therefore, if $a < a' < b$, we must have
\[
\P \left[\{N \le \exp(e^{a/p}) \} \cap \{K \ge e^{a'/p}\} \right] \to 0, \quad \text{as} \;p \to 0.
\]
Since $\P [K \ge e^{a'/p} ]  \to 1$ by Theorem \ref{thm:KL}, this implies that $\P[N \le \exp(e^{a/p}) ]) \to 0$.

If $\beta > \alpha$, the probability in \eqref{alphabetalim} goes to $1$, as $p \to 0$.  This means that if $c>c'>b$, we must have
\[
\P \left[M_n \ge e^{c'/p}, \;\; n = \exp(e^{c/p}) \right] \to 1, \quad \text{as} \;p \to 0. 
\]
On the other hand, the event $\{ K \geq e^{c'/p}\}$ contains the event $\{ M_n \ge e^{c'/p}, \;\; n = \exp(e^{c/p}) \} \cap \{ N \ge e^{c/p}\}$.  Since $\P[ K \geq e^{c'/p} ] \to 0$, by Theorem \ref{thm:KU}, we conclude that $\P[ N \ge e^{c/p}] \to 0$.
\end{proof}

\subsection{Proof of Theorem \ref{prop:stretch}} \label{sec:stretch}

\begin{proof}
Scaling time to eliminate the normalizing constant $C_\alpha$, we can assume that rain lands on $k$ 
at rate $e^{-k^\alpha} - e^{-(k+1)^\alpha}$ for $k=1, 2, \ldots$.
When the maximal wet site is at $M_t$, let $\overline H_t$ be the number of dry sites
in $[M_t-\ell(t),M_t))$ where  $\ell(t) = M_t^{1-\alpha}/2$. We divide by 2 so that if at time $t$ there is a jump of more than $k^{1-\alpha}$,
which will increase the size of the viewing window, all of the sites within $\ell(t)$ of the boundary are vacant. The first block size $K$ is finite only if $\bar H_t$ hits zero in finite time.

As before to simplify the arithmetic we run time at rate $\exp(M_t^\alpha)$. When we do this, jumps of $M_t$ larger than $2\ell(t)$ occur 
at rate $e^{- (M_t + 2\ell(t))^\alpha} e^{M_t^\alpha}  \approx e^{-\alpha}$ since when  $j \ll k$ we have
$$
(k+j)^{\alpha} - k^{\alpha} = \int_k^{k+j} \alpha x^{\alpha-1} \, dx \approx \alpha k^{\alpha-1}j,
$$
which is approximately $\alpha$ when $j = k^{1-\alpha}$ and $k$ is large.
 Such jumps in $M$ reset $\bar H$ to its maximal value of $M^{1-\alpha}/2$. 

Using the same reasoning as in the proof of Theorem \ref{thm:KL}, we see that $H_t$ jumps from $i$ to $i-1$ at a rate no more than  $r_i = e^{\alpha}i/\ell(t) + i/\ell(t)$. If $\epsilon = e^{-\alpha}/(e^\alpha+1)$ then when $i/\ell(t) < \epsilon$ we have $r_i < e^{-\alpha}$. So, given that the maximal wet site is at $M_t = k$ and $\bar H_t = \ell(t) =  k^{1-\alpha}/2$ (its maximum possible value), the probability that $\bar H$ becomes $0$ before resetting to its maximum value is
$\le \sigma_k = (1/2)^{\epsilon \ell(t)} = (1/2)^{\epsilon k^{1-\alpha}/2}$. Since $\sum_k \sigma_k < \infty$, it follows that with positive probability $\bar H$ will reset to its maximum value infinitely many times before hitting zero. Thus, $\P[K = \infty]>0$. 
	
\end{proof}

\section*{Acknowledgements}
Thanks to Noah Forman for sharing a continuous version of this problem that he called the ``whack-a-mole" process, and to Jim Pitman for helpful correspondence. I. Cristali, V. Ranjan and J. Steinberg were undergraduates participating in the Summer 2017 Duke Opportunities in Math program, partially supported by NSF grant DMS 1406371, and supervised by the last four authors.  M.\ Junge was partially supported by NSF grant DMS 1614838, R.\ Durrett  by DMS 1505215, and J.\ Nolen and E.\ Beckman by DMS 1351653.

\bibliographystyle{amsalpha}
\bibliography{rainstick}

\newcommand{\etalchar}[1]{$^{#1}$}
\providecommand{\bysame}{\leavevmode\hbox to3em{\hrulefill}\thinspace}
\providecommand{\MR}{\relax\ifhmode\unskip\space\fi MR }
\providecommand{\MRhref}[2]{%
  \href{http://www.ams.org/mathscinet-getitem?mr=#1}{#2}
}
\providecommand{\href}[2]{#2}
\begin{thebibliography}{GIN{\etalchar{+}}09}

\bibitem[BB16]{naya}
Riddhipratim Basu and Nayantara Bhatnagar, \emph{Limit theorems for longest
  monotone subsequences in random \text{M}allows permutations}, arXiv preprint
  arXiv:1601.02003 (2016).

\bibitem[BHJ92]{poisson}
A.D. Barbour, L.~Holst, and S.~Janson, \emph{Poisson approximation}, Oxford
  science publications, Clarendon Press, 1992.

\bibitem[DPT17]{pitman2}
J.-J. {Duchamps}, J.~{Pitman}, and W.~{Tang}, \emph{{Renewal sequences and
  record chains related to multiple zeta sums}}, ArXiv e-prints (2017).

\bibitem[FMS96]{geo1}
James~Allen Fill, Hosam~M. Mahmoud, and Wojciech Szpankowski, \emph{On the
  distribution for the duration of a randomized leader election algorithm}, The
  Annals of Applied Probability \textbf{6} (1996), no.~4, 1260--1283.

\bibitem[G{\etalchar{+}}04]{BS2}
Alexander~V Gnedin et~al., \emph{The bernoulli sieve}, Bernoulli \textbf{10}
  (2004), no.~1, 79--96.

\bibitem[GIM10]{bs3}
Alexander Gnedin, Alexander Iksanov, and Alexander Marynych, \emph{Limit
  theorems for the number of occupied boxes in the bernoulli sieve}.

\bibitem[GIN{\etalchar{+}}09]{bs4}
Alexander~V Gnedin, Alexander~M Iksanov, Pavlo Negadajlov, Uwe R{\"o}sler,
  et~al., \emph{The bernoulli sieve revisited}, The Annals of Applied
  Probability \textbf{19} (2009), no.~4, 1634--1655.

\bibitem[GO10]{gnedin}
Alexander Gnedin and Grigori Olshanski, \emph{q-exchangeability via
  quasi-invariance}, The Annals of Probability (2010), 2103--2135.

\bibitem[JS97]{geo2}
Svante Janson and Wojciech Szpankowski, \emph{Analysis of an asymmetric leader
  election algorithm}, the electronic journal of combinatorics \textbf{4}
  (1997), no.~1, 17.

\bibitem[MIG10]{BS}
Alexander Marynych, Alexander Iksanov, and Alexander Gnedin, \emph{The
  bernoulli sieve: an overview}, Discrete Mathematics \& Theoretical Computer
  Science (2010).

\bibitem[PT17]{pitman}
Jim {Pitman} and Wenpin {Tang}, \emph{{Regenerative random permutations of
  integers}}, ArXiv e-prints: 1704.01166 (2017).

\end{thebibliography}

\end{document}